\documentclass[11pt]{amsart}
\usepackage{amsmath,amssymb,amsthm,enumerate}

\usepackage{xy} 
\usepackage{mathrsfs} 
\usepackage{xspace}

\title[Finitely generated branch groups]{Commensurated subgroups in finitely generated branch groups}
\author{Phillip Wesolek}
\address{ Universit\'{e} catholique de Louvain,
	Institut de Recherche en Math\'{e}matiques et Physique (IRMP),
	Chemin du Cyclotron 2, box L7.01.02,
	1348 Louvain-la-Neuve, Belgium}
\email{phillip.wesolek@uclouvain.be}

\newtheorem{thm}{Theorem}[section]

\newtheorem{prop}[thm]{Proposition}

\newtheorem{cor}[thm]{Corollary}

\theoremstyle{definition}
\newtheorem{defn}[thm]{Definition}

\newtheorem*{ack}{Acknowledgments}
\newtheorem{ex}[thm]{Example}

\newcommand{\Zb}{\mathbb{Z}}
\newcommand{\Nb}{\mathbb{N}}

\newcommand{\mf}[1]{\mathfrak{#1}}

 \newcommand{\rist}{\mathrm{rist}}
 \newcommand{\Aut}{\mathrm{Aut}}
 \newcommand{\ssl}{/\!\!/}
 \newcommand{\normal}{\trianglelefteq}
 \newcommand{\sym}{\mathrm{Sym}}
 
 \newcommand{\tdlc}{t.d.l.c.\@\xspace}
 \newcommand{\tdlcsc}{t.d.l.c.s.c.\@\xspace}

\newcommand{\grp}[1]{\langle #1 \rangle}
\newcommand{\ol}[1]{\overline{#1}}

\begin{document}

\begin{abstract}
A subgroup $H\leq G$ is commensurated if $|H:H\cap gHg^{-1}|<\infty$ for all $g\in G$. We show that a finitely generated branch group is just infinite if and only if every commensurated subgroup is either finite or of finite index. As a consequence, every commensurated subgroup of the Grigorchuk group and many other branch groups of independent interest are either finite or of finite index.
\end{abstract}
\maketitle

\section{Introduction}
Subgroups $H$ and $K$ of a group $G$ are \textbf{commensurate} if $|H:H\cap K|<\infty$ and $|K:H\cap K|<\infty$. The subgroup $H$ is \textbf{commensurated} in $G$ if $H$ and $gHg^{-1}$ are commensurate for all $g\in G$. Normal subgroups are obvious examples of commensurated subgroups. However, commensurated subgroups need not be even commensurate with a normal subgroup. Simple groups can admit such commensurated subgroups; for example, Thompson's group $V$ admits an infinite commensurated proper subgroup.

G. Margulis' celebrated normal subgroup theorem demonstrates that any lattice in a higher rank simple algebraic group is just infinite - i.e.\ every non-trivial normal subgroup is of finite index; see \cite{Mar91}. Margulis and R. Zimmer then ask if the commensurated subgroups can be classified up to commensurability by a precise family of known commensurated subgroups; this question is sometimes called the \textit{Margulis--Zimmer commensurated subgroup problem}.  Aside from a strengthening of the normal subgroup theorem, the commensurated subgroup problem seems related to many aspects of arithmetic groups, as discussed in \cite{SW13}. In loc.\ cit.,  Y. Shalom and G. Willis classify the commensurated subgroups for a large family of arithmetic groups, making substantial progress on this problem.

Considering the analogues of the commensurated subgroup problem for other classes of groups with few normal subgroups seems independently interesting. Using the Shalom--Willis strategy of studying certain completions, we here classify the commensurated subgroups of finitely generated just infinite branch groups. Indeed, we characterize the just infinite property for finitely generated branch groups by commensurated subgroups. 

\begin{thm}\label{thm:main}
Suppose $G$ is a finitely generated branch group. Then $G$ is just infinite if and only if every commensurated subgroup is either finite or of finite index.
\end{thm}
As an immediate consequence of this result, we obtain a description of the commensurated subgroups of various groups of independent interest. 
\begin{cor}The Grigorchuk group and the Gupta--Sidki groups are such that every commensurated subgroup is either finite or of finite index.
\end{cor}

\section{Preliminaries}
We use ``t.d.", ``l.c.", and ``s.c." for ``totally disconnected", ``locally compact", and ``second countable", respectively.  Recall that a \tdlcsc group is a \textbf{Polish group} - i.e. it is separable and admits a complete, compatible metric.

\subsection{Branch groups}
Our approach to branch groups follows closely R. I. Grigorchuk's presentation in \cite{G00}.  

A \textbf{rooted tree} $T$ is a locally finite tree with a distinguished vertex $r$ called the \textbf{root}. Letting $d$ be the usual graph metric, the \textbf{levels} of $T$ are the sets $V_n:=\{v\in T\mid d(v,r)=n\}$. The \textbf{degree} of a vertex $v\in V_n$ is the number of $w\in V_{n+1}$ such that there is an edge from $v$ to $w$. When vertices $k$ and $w$ lie on the same path to the root and $d(k,r)\leq d(w,r)$, we write $k\leq w$. Given a vertex $s\in T$, the \textbf{tree below $s$}, denoted by $T^s$, is the collection of $t$ such that $s\leq t$ along with the induced graph structure.

We call a rooted tree \textbf{spherically homogeneous} if all $v,w\in V_n$ the degree of $v$ is the same as the degree of $w$. A spherically homogeneous tree is completely determined by specifying the degree of the vertices at each level. These data are given by an infinite sequence $\alpha\in \Nb^{\Nb}$ such that $\alpha(i)\geq 2$ for all $i\in \Nb$. The condition $\alpha(i)\geq 2$ is to ensure that levels are not redundant; i.e. if $\alpha(i)=1$, then we can remove the $i$-th level without changing the automorphism group. We denote a spherically homogeneous tree by $T_{\alpha}$ for $\alpha\in \Nb_{\geq 2}^{\Nb}$. When $\alpha\equiv k$ for some $k\geq 2$, we write $T_{k}$.

For $G\leq \Aut(T_{\alpha})$ a subgroup and for a vertex $v\in T_{\alpha}$, the \textbf{rigid stabilizer of $v$} in $G$ is defined to be
\[
\rist_G(v):=\{g\in G\mid g(w)=w \text{ for all }w\in T_{\alpha}\setminus T_{\alpha}^v\}.
\]
The \textbf{$n$-th rigid level stabilizer} in $G$ is defined to be
\[
\rist_G(n):=\grp{\rist_G(v)\mid v\in V_n}.
\]
It is easy to see that $\rist_G(n)\simeq \prod_{v\in V_n}\rist_G(v)$.

\begin{defn}\label{df:profinitebranch}
A group $G$ is said to be a \textbf{branch group} if there is a rooted tree $T_{\alpha}$ for some $\alpha\in \Nb_{\geq 2}^{\Nb}$ such that the following hold:
\begin{enumerate}[(i)]
\item $G$ is isomorphic to a subgroup of $\Aut(T_{\alpha})$. 
\item $G$ acts transitively on each level of $T_{\alpha}$.
\item For each level $n$, the index $|G:\rist_G(n)|$ is finite.
\end{enumerate}
\end{defn}
\noindent Let $\mathrm{(iii)'}$ be the the following condition: Every $\rist_G(n)$ is infinite. A group satisfying $\mathrm{(i)}$, $\mathrm{(ii)}$, and $\mathrm{(iii)'}$ is called a \textbf{weakly branch} group. Plainly, every branch group is also weakly branch.

An infinite group is \textbf{just infinite} if all proper quotients are finite. Just infinite branch groups already have a characterization in terms of certain \textit{normal} subgroups.

\begin{thm}[Grigorchuk, {\cite[Theorem 4]{G00}}]\label{thm:ji}
Suppose $G\leq \Aut(T_{\alpha})$ is a branch group. Then $G$ is just infinite if and only if the commutator subgroup $\rist_G(k)'$ has finite index in $\rist_G(k)$ for all levels $k$.
\end{thm}

We shall need a fact implicit in the proof of \cite[Theorem 4]{G00}.
\begin{prop}[Grigorchuk]\label{prop:contain_comm}
Suppose $G\leq \Aut(T_{\alpha})$ acts transitively on each level of $T_{\alpha}$. If $H\normal G$ is non-trivial, then there is a level $m$ such that $\rist_G(m)'\leq H$.
\end{prop}

\subsection{Completions and chief blocks}
Our proof requires the Schlichting completion. This completion has appeared in various contexts in the literature. See for example \cite{T03} or consider \cite{RW_Hom_15} for a longer discussion. We here give a brief account

Given a countable group $G$ with a commensurated subgroup $O$, the group $G$ acts by left multiplication on the collection of left cosets $G/O$. This induces a permutation representation $\sigma:G\rightarrow \sym(G/O)$ with kernel the normal core of $O$ in $G$. The group $\sym(G/O)$ is a topological, indeed Polish, group under the pointwise convergence topology, and we may thus form a completion as follows:
 
\begin{defn}
For a countable group $G$ with a commensurated subgroup $O$, the \textbf{Schlichting completion} of $G$ with respect to $O$, denoted by $G\ssl O$, is defined to be $\ol{\sigma(G)}$. The map $\sigma:G\rightarrow G\ssl O$ is called the \textbf{completion map}.
\end{defn}
It is easy to verify that $G\ssl O$ is a \tdlcsc group. When $G$ is finitely generated,  $G\ssl O$ is additionally compactly generated.

We shall also need the theory of chief blocks developed in \cite{RW_P_15}. A \textbf{normal factor} of a topological group $G$ is a quotient $K/L$ such that $K$ and $L$ are closed normal subgroups of $G$ with $L < K$. We say that $K/L$ is a \textbf{chief factor} if there are no closed normal subgroups of $G$ strictly between $L$ and $K$. The centralizer of a normal factor $K/L$ is 
\[
C_G(K/L):=\{g\in G\mid [g,K]\leq L\}.
\]

Centralizers give a notion of equivalence for chief factors; we restrict this equivalence to non-abelian chief factors for technical reasons. Non-abelian chief factors $K_1/L_1$ and $K_2/L_2$ are \textbf{associated} if $C_G(K_1/L_1)=C_G(K_2/L_2)$. For a non-abelian chief factor $K/L$, the equivalence class of non-abelian chief factors equivalent to $K/L$ is denoted by $[K/L]$. The class $[K/L]$ is called a \textbf{chief block} of $G$. The set of chief blocks of $G$ is denoted  by $\mf{B}_G$. For a chief block $\mf{a}$, the centralizer $C_G(\mf{a})$ is defined to be $C_G(K/L)$ for some (equivalently, any) representative $K/L$.

A key property of chief blocks is a general refinement theorem.

\begin{thm}[Reid--Wesolek, {\cite[Theorem 1.15]{RW_P_15}}]\label{thmintro:Schreier_refinement}
Let $G$ be a Polish group, $\mf{a}\in \mf{B}_G$, and
\[
\{1\} = G_0 \le G_1 \le \dots \le G_n = G
\]
be a series of closed normal subgroups in $G$.  Then there is exactly one $i \in \{0,\dots,n-1\}$ such that there exist closed normal subgroups $G_i \le B < A \le G_{i+1}$ of $G$ for which $A/B\in \mf{a}$.  
\end{thm}

\section{Commensurated subgroups}
We first establish the reverse implication of our main theorem. For this implication, we need not assume the group is finitely generated, and the result holds for weakly branch groups.
\begin{prop}\label{prop:reverse}
Let $G\leq \Aut(T_{\alpha})$ be a weakly branch group. If every commensurated subgroup of $G$ is either finite or of finite index, then $G$ is just infinite.
\end{prop}
\begin{proof}
Fix a level $m$. The commutator subgroup $H:=\rist_G(m)'$ is then a normal subgroup of $G$ and, a fortiori, commensurated. Suppose for contradiction that $H$ is finite. Since $\bigcap_{n\geq m}\rist_G(n)=\{1\}$, there is some $k\geq m$ such that $\rist_G(k)\cap H=\{1\}$. The group $\rist_G(k)$ then injects into $\rist_G(m)/H$, so it is abelian. This is absurd since weakly branch groups do not admit abelian rigid stabilizers. We thus deduce that $H$ is of finite index in $G$, and it follows that $G$ is a branch group. Appealing to Theorem~\ref{thm:ji}, $G$ is just infinite. 
\end{proof}

We now consider the converse for finitely generated branch groups.
\begin{thm}\label{thm:forward}
Suppose $G\leq \Aut(T_{\alpha})$ is a finitely generated branch group. If $G$ is just infinite, then every commensurated subgroup of $G$ is either finite or of finite index.
\end{thm}
\begin{proof}
Suppose for contradiction $O\leq G$ is an infinite commensurated subgroup of infinite index. Form the Schlichting completion $H:=G\ssl O$ and let $\sigma:G\rightarrow H$ be the completion map.

Since $G$ is finitely generated, $H$ is a compactly generated \tdlcsc group. For all open normal subgroups $L\normal H$, the preimage $\sigma^{-1}(L)$ is a non-trivial normal subgroup of $G$, hence it has finite index. It follows that $L$ is a finite index open subgroup of $H$. Every open normal subgroup of $H$ therefore has finite index. Appealing to \cite[Theorem F]{CM11}, we deduce that 
\[
R:=\bigcap\{O\normal H\mid O\text{ is open}\}
\]
is a cocompact characteristic subgroup of $H$ without non-trivial discrete quotients. If $R$ is trivial, then $H$ is a compact group, and $O$ has finite index in $G$. However, this is absurd, as we assume $O$ has infinite index.

The group $R$ is thus an infinite compactly generated \tdlcsc group with no non-trivial discrete quotients. Since $R$ is \tdlc, any non-trivial compact quotient is profinite, and thus, such a quotient produces a non-trivial discrete quotient. We deduce that $R$ additionally has no non-trivial compact quotient. The result \cite[Theorem A]{CM11} now implies that $R$ admits exactly $n$ non-discrete topologically simple quotients where $0<n<\infty$; say that $N_1,\dots,N_n$ lists the kernels of these quotients. The group $H$ acts on $\{N_1,\dots,N_n\}$ by conjugation, so there is a closed $\tilde{H}\normal H$ with finite index such that $\tilde{H}$ fixes each $N_i$. The pre-image $\sigma^{-1}(\tilde{H})$ is then a finite index normal subgroup of $G$. Via Proposition~\ref{prop:contain_comm}, there is some level $m$ of the tree such that $\rist_G(m)'\leq \sigma^{-1}(\tilde{H})$, and we may assume $m>n$. Taking $E:=\ol{\sigma(\rist_G(m)')}$, we have that $E$ is a finite index subgroup of $H$ and that $E$ normalizes each $N_i$. Each factor $R/N_i$ is thus a chief factor of $E$; let $\mf{a}_i$ be the chief block of $E$ given by $R/N_i$. 

For each $v\in V_m$, the subgroup $L_v:=\ol{\sigma(\rist_G(v)')}$ is a non-abelian closed normal subgroup of $E$. Letting $v_1,\dots,v_k$ list $V_m$, put $K_i:=\ol{L_{v_1}\dots L_{v_i}}$ and observe that $K_i<K_{i+1}$. We thus obtain a normal series for $E$:
\[
\{1\}<K_1<\dots<K_k=E.
\]
Repeatedly applying Theorem~\ref{thmintro:Schreier_refinement}, we may refine the series to include a representative for each $\mf{a}_i$. Since $n<k$, there is some $K_j<K_{j+1}$ such that the refinement puts no subgroups between $K_j$ and $K_{j+1}$. The subgroup $L_{v_{j+1}}$ is plainly contained in the centralizer of any $\mf{a}_l$ which has a representative which appears in the refined series after $K_{j+1}$. On the other hand, since $L_{v_{j+1}}$ centralizes $K_{j}$, the group $L_{v_{j+1}}$ also centralizes the $\mf{a}_l$ with representatives appearing in the refined series before $K_{j}$. The group $L_{v_{j+1}}$ thus centralizes each block  $\mf{a}_1,\dots,\mf{a}_n$, and therefore, it centralizes each factor $R/N_i$.

Returning to the setting of $H$, the subgroup $K:=\bigcap_{i=1}^nC_H(R/N_i)$ is normal in $H$, and moreover, the previous paragraph ensures $L_{v_{j+1}}\leq K$. Therefore, $K$ intersects $\sigma(G)$ non-trivially, so $\sigma^{-1}(K)$ is a finite index subgroup of $G$. We conclude that $K$ has finite index in $H$, so $R\leq K$. Each $R/N_i$ is thus abelian, which is absurd.
\end{proof}

\begin{proof}[Proof of Theorem~\ref{thm:main}] Proposition~\ref{prop:reverse} gives the reverse implication. Theorem~\ref{thm:forward} gives the forward implication.
\end{proof}

The following example shows that the finite generation hypothesis is necessary in Theorem~\ref{thm:forward}:

\begin{ex} Let $A_5$ be the alternating group on five elements and take the usual permutation representation $(A_5,[5])$. For each $n\geq 1$, let $K_n$ be the iterated wreath product of $n$ copies of $(A_5,[5])$. The permutation group given by the imprimitive action $(K_n,[5]^n)$ induces an embedding $\phi_n:K_n\rightarrow \Aut(T_5)$, where the action of $K_n$ on $T_5$ moves the vertices below the $n$-th level rigidly.
	
The $K_n$ form a directed system, so we may take the direct limit $G$. One verifies that the maps $\phi_n$ cohere to induce a map $\phi:G\rightarrow \Aut(T_5)$. The map $\phi$ moreover witnesses that $G$ is a branch group. Applying Theorem~\ref{thm:ji}, it follows that $G$ is also just infinite. 
	
However, $G$ admits infinite commensurated subgroups of infinite index. For example, let $F$ be a proper non-trivial subgroup of $A_5$ and let $(F,[5])$ be the permutation representation induced by $(A_5,[5])$. The iterated wreath products of copies of $(F,[5])$ again form a direct system. Moreover, the direct limit is an infinite commensurated subgroup of infinite index in $G$.	
\end{ex}

We conclude with an easy observation. The results of Shalom--Willis \cite{SW13} show that various arithmetic groups, including $SL_n(\Zb)$ for $n\geq 3$, have the following strong property, which is sufficient to ensure every commensurated subgroup is either finite or of finite index for a just infinite group. We say that $K\leq G$ \textbf{commensurates} $H\leq G$ if $|H:H\cap kHk^{-1}|<\infty$ for all $k\in K$.
\begin{defn}
A group $G$ is said to have the \textbf{outer commensurator-normalizer property} if the following holds: for every group $H$ and every homomorphism $\psi:G\rightarrow H$, if there is $D\leq H$ commensurated by $\psi(G)$, then there is $\widetilde{D}\leq H$ commensurate with $D$ and normalized by $\psi(G)$. 
\end{defn}

Just infinite finitely generated branch groups can fail the outer commensurator-normalizer property. To see this, we recall a standard group-theoretic construction:

\begin{defn} Suppose $(G_i)_{i\in \Nb}$ is a sequence of \tdlc groups and suppose there is a distinguished compact open subgroup $O_i\leq G_i$ for each $i\in \Nb$. The \textbf{local direct product} of $(G_i)_{i\in \Nb}$ over $(O_i)_{i\in \Nb}$ is defined to be 
	\[
	\left\{ f:\Nb\rightarrow \bigsqcup_{i\in \Nb} G_i\mid f(i)\in G_i \text{, and }f(i)\in O_i \text{  for all but finitely many }i\in \Nb\right\}
	\]
	with the group topology such that $\prod_{i\in \Nb}O_i$ continuously embeds as an open subgroup. We denote the local direct product by $\bigoplus_{i\in \Nb}\left(G_i,O_i\right)$.
\end{defn}
Local direct products of \tdlc groups are again \tdlc groups.

\begin{prop} 
The Grigorchuk group fails the outer commensurator-normalizer property.
\end{prop}
\begin{proof}
The Grigorchuk group $G$ admits an action on a countable set $X$ with a non-transfixed commensurated subset $Y\subseteq X$. That is to say, there is no $Y'\subseteq X$ such that $G$ fixes $Y'$ setwise and $|Y\Delta Y'|<\infty$; see \cite[Section 2.7]{C13}.

Fix a non-trivial finite group $F$, let $F_x$ list copies of $F$ indexed by $X$, and for each $x\in X$, define 
\[
U_x:=\begin{cases} F &\mbox{if } x\in Y \\ 
\{1\} & \mbox{else} . 
\end{cases}
\] 
We then form the local direct product $\bigoplus_{x\in X}(F_x,U_x)$. The group $G$ obviously acts on $\bigoplus_{x\in X}(F_x,U_x)$ by shift, so we take
\[
H:=\bigoplus_{x\in X}(F_x,U_x)\rtimes G.
\]
Since $Y$ is a commensurated subset of $X$ and $F$ is finite, it follows that $H$ is a \tdlc group.

The subgroup $U:=\prod_{x\in Y}F_x$ is a compact open subgroup of $H$, and thus, $G$ commensurates it. Suppose for contradiction that $V\leq H$ is commensurate with $U$ and normalized by $G$.  Passing to the closure if necessary, we may take $V$ to be closed, hence $V$ is also a compact open subgroup. Let $Y'$ be the collection of coordinates $x$ such that the projection $\pi_x(V)$ is non-trivial. It follows that $|Y'\Delta Y|<\infty$. However, since $G$ normalizes $V$, the set $Y'$ must be stabilized by $G$, contradicting our choice of $Y$.
\end{proof}

\begin{ack}
	The author thanks Colin Reid for his helpful remarks and the anonymous referee for his or her detailed suggestions. The author was supported by ERC grant \#278469.
\end{ack}


\bibliographystyle{amsplain}
\bibliography{biblio1}

\end{document}